\documentclass[11pt]{amsart}

\usepackage{amsmath, amssymb, amscd, cancel, graphicx, soul, stmaryrd}

\usepackage{mathdots}

\usepackage[dvipsnames,usenames]{color}
\usepackage[colorlinks=true, urlcolor=NavyBlue, linkcolor=NavyBlue, citecolor=NavyBlue]{hyperref}
\usepackage{float}
\usepackage{amsmath, amsthm, amssymb}
\usepackage{amsfonts}
\usepackage{enumerate}
\usepackage{paralist}
\usepackage{xypic}
\usepackage{subfigure}
\usepackage{hyperref}
\usepackage{color}
\usepackage{marginnote}

\hypersetup{
linkbordercolor={1 0 0}, 
citebordercolor={0 1 0} 
}
\newtheorem{thm}{Theorem}[section]

\newtheorem{cor}[thm]{Corollary}
\newtheorem{lem}[thm]{Lemma}
\newtheorem{prop}[thm]{Proposition}

\newtheorem{defin}[thm]{Definition}

\def\bC{{\mathbb C}}

\def\cM{{\mathcal M}}

\def\bQ{{\mathbb Q}}
\def\bR{{\mathbb R}}
\def\bT{{\mathbb T}}
\def\fT{{\mathfrak T}}

\def\bZ{{\mathbb Z}}

\def\spc{{\mathrm{spin^c}}}
\def\Spc{{\mathrm{Spin^c}}}

\def\del{{\partial}}

\def\rk{{\mathrm{rk}}}

\def\Spc{{\mathrm{Spin}^c}}
\def\wa{{\widetilde \alpha}}
\def\wb{{\widetilde \beta}}
\def\wg{{\widetilde \gamma}}

\begin{document}

\title[$(1,1)$ L-space knots]%
{$(1,1)$ L-space knots}

\author{Joshua Evan Greene}
\address{Department of Mathematics, Boston College\\ Chestnut Hill, MA 02467}
\email{joshua.greene@bc.edu}

\author{Sam Lewallen}
\address{Princeton Neuroscience Institute, Princeton University \\ Princeton, NJ 08544}
\email{lewallen@princeton.edu}

\author{Faramarz Vafaee}
\address{Mathematics Department, California Institute of Technology \\ Pasadena, CA 91125}
\email{vafaee@caltech.edu}

\maketitle

\medskip

\noindent {\bf Abstract.} 
We characterize the $(1,1)$ knots in the three-sphere and lens spaces that admit non-trivial L-space surgeries.
As a corollary, 1-bridge braids in these manifolds admit non-trivial L-space surgeries.
We also recover a characterization of the Berge manifold amongst 1-bridge braid exteriors.
\medskip


\section{Introduction.}

An \emph{L-space} is a rational homology sphere $Y$ with the ``simplest" Heegaard Floer invariant: $\widehat{HF}(Y)$ is a free abelian group of rank $|H_1(Y; \bZ)|$.
Examples abound and include lens spaces and, more generally, connected sums of manifolds with elliptic geometry \cite{Ath}.
One of the most prominent problems in relating Heegaard Floer homology to low-dimensional topology is to give a topological characterization of L-spaces.
Work by many researchers has synthesized a bold and intriguing proposal that seeks to do so in terms of taut foliations and orderability of the fundamental group \cite[Conjecture 5]{juhaszsurvey}.

A prominent source of L-spaces arises from surgeries along knots.
Suppose that $K$ is a knot in a closed three-manifold $Y$.
If $K$ admits a non-trivial surgery to an L-space, then $K$ is an {\em L-space knot}.
Examples include torus knots and, more generally, Berge knots in $S^3$ \cite{Berge}; two more constructions especially pertinent to our work appear in \cite{HomLidmanVafaee,Vafaee2013}.
If an L-space knot $K$ admits more than one L-space surgery -- for instance, if $Y$ itself is an L-space -- then it admits an interval of L-space surgery slopes, so it generates abundant examples of L-spaces \cite{Rasmussen2015floer}.
With the lack of a compelling guiding conjecture as to which knots are L-space knots, and as a probe of the L-space conjecture mentioned above, it is valuable to catalog which knots in various special families are L-space knots. This is the theme of the present work.

The manifolds in which we operate are the rational homology spheres that admit a genus one Heegaard splitting, namely the three-sphere and lens spaces.
The knots we consider are the {\em $(1,1)$ knots} in these spaces: these are the knots that can be isotoped to meet each Heegaard solid torus in a properly embedded, boundary-parallel arc. 
Our main result, Theorem \ref{thm:main} below, characterizes $(1,1)$ L-space knots in simple, diagrammatic terms.

A {\em $(1,1)$ diagram} is a doubly-pointed Heegaard diagram $(\Sigma, \alpha, \beta, z, w)$, where 
$(\Sigma, \alpha, \beta)$ is a genus one Heegaard diagram of a 3-manifold $Y$.
The $(1,1)$ knots in $Y$ are precisely those that admit a doubly-pointed Heegaard diagram  \cite{Goda2005, Hedden2011Berge,Rasmussen2005}.
A $(1,1)$ diagram is {\em reduced} if every bigon contains a basepoint.
We can transform a given $(1,1)$ diagram of $K$ into a reduced $(1,1)$ diagram of $K$ by isotoping the curves into minimal position in the complement of the basepoints: we accomplish this by successively isotoping away bigons in the complement of the basepoints and the curves in $\Sigma$. 
Our characterization of $(1,1)$ L-space knots in $S^3$ and lens spaces is expressed in terms of the following property of $(1,1)$ diagrams:

\begin{defin}
\label{defin:Coherent}
A reduced $(1,1)$ diagram $(\Sigma,\alpha,\beta,z,w)$ is {\bf coherent} if there exist orientations on $\alpha$ and $\beta$ that induce coherent orientations on the boundary of every embedded bigon $(D,\del D) \subset (\Sigma, \alpha \cup \beta)$.
Its sign, positive or negative, is the sign of $\alpha \cdot \beta$, with these curves coherently oriented.
\end{defin}

\noindent
Coherence is easy to spot in a diagram: see Figures \ref{fig:examples} and \ref{fig:pqrs} and the second paragraph of Subsection \ref{ss:11}.

We may now state the main result of the paper:

\begin{thm}
\label{thm:main}
A reduced $(1,1)$ diagram presents an L-space knot if and only if it is coherent.
The knot is a positive or negative L-space knot according to the sign of the coherent diagram.
\end{thm}

\noindent
The sign of an L-space knot is the sign of an L-space surgery slope along it, which we review in Subsection \ref{ss:HFK}.
Note that a given knot may admit non-homeomorphic coherent $(1,1)$ diagrams.
However, Theorem \ref{thm:main} implies that its reduced $(1,1)$ diagrams are either all incoherent or else all coherent and of the same sign. 
Again, see Figure \ref{fig:examples}.

We apply Theorem \ref{thm:main} to show that a broad family of $(1,1)$ knots are L-space knots.
We recall the following construction, first studied by Berge~\cite{Berge1991} and Gabai~\cite{Gabai1990}, and state a natural generalization of it.

\begin{defin}
\label{def:1-bridge1}
A knot in the solid torus $S^1 \times D^2$ is a {\bf 1-bridge braid} if it is isotopic to a union of two arcs $\gamma \cup \delta$ such that
\begin{itemize}
\item
$\gamma \subset \del (S^1 \times D^2)$ is {\em braided}, i.e., transverse to each meridian $\mathrm{pt.} \times \del D^2$, and
\item
$\delta$ is a {\em bridge}, i.e., properly embedded in some meridional disk $\mathrm{pt.} \times D^2$.
\end{itemize}
It is {\em positive} if $\gamma$ is a positive braid in the usual sense.
A knot in a closed three-manifold $Y$ with a genus one Heegaard splitting is a {\em 1-bridge braid} if it is isotopic to a 1-bridge braid supported within one of the Heegaard solid tori.
\end{defin}

\noindent
J. and S. Rasmussen conjectured that a positive 1-bridge braid in $S^3$ is a positive L-space knot at the end of \cite{Rasmussen2015floer}. We prove a generalization of their conjecture in Theorem \ref{thm:1-bridgerefinement}.  Without the sign refinement, the result reads:

\begin{thm}
\label{thm:1-Bridge}
1-bridge braids in $S^3$ and lens spaces are L-space knots.
\end{thm}

\noindent
Krcatovich has found many examples of $(1,1)$ L-space knots in $S^3$ that are not 1-bridge braids by a computer search \cite{Krcatovich2016private}.
A small representative is the knot $K(21,4,4,11)$ in the notation of \cite{Rasmussen2005}.
He showed that it is an L-space knot by an application of Theorem \ref{thm:main}, and that its Alexander polynomial distinguishes it from 1-bridge braids by comparing with a list tabulated by J. Rasmussen.

We collect the necessary background on Heegaard Floer homology and prove Theorem \ref{thm:main} in Section \ref{sec:Background}.
We prove Theorem \ref{thm:1-Bridge} and its sign-refined version, Theorem \ref{thm:1-bridgerefinement}, in Section \ref{sec:1-bridge}.
We also use Theorem \ref{thm:1-bridgerefinement} to characterize the Berge manifold in Proposition \ref{prop:unique}.

We leave open two natural problems: the isotopy classification of $(1,1)$ L-space knots, and the determination of whether surgeries along these knots conform to \cite[Conjecture 5]{juhaszsurvey}.


\section*{Acknowledgments.}
We thank Matt Hedden, Jen Hom, David Krcatovich, Adam Levine, Clayton McDonald, Yi Ni, and Alex Zupan for helpful conversations.
We thank the referee for a very thorough and thoughtful review.
Sam Lewallen would like to thank Zolt\'an Szab\'o for suggesting he study $(1,1)$ L-space knots, and both Zolt\'an and Liam Watson for their interest and encouragement. 
JEG was supported by NSF CAREER Award DMS-1455132 and an Alfred P. Sloan Foundation Research Fellowship.
SL was supported by an NSF Graduate Research Fellowship.  


\begin{figure}[t]
\includegraphics[width=4in]{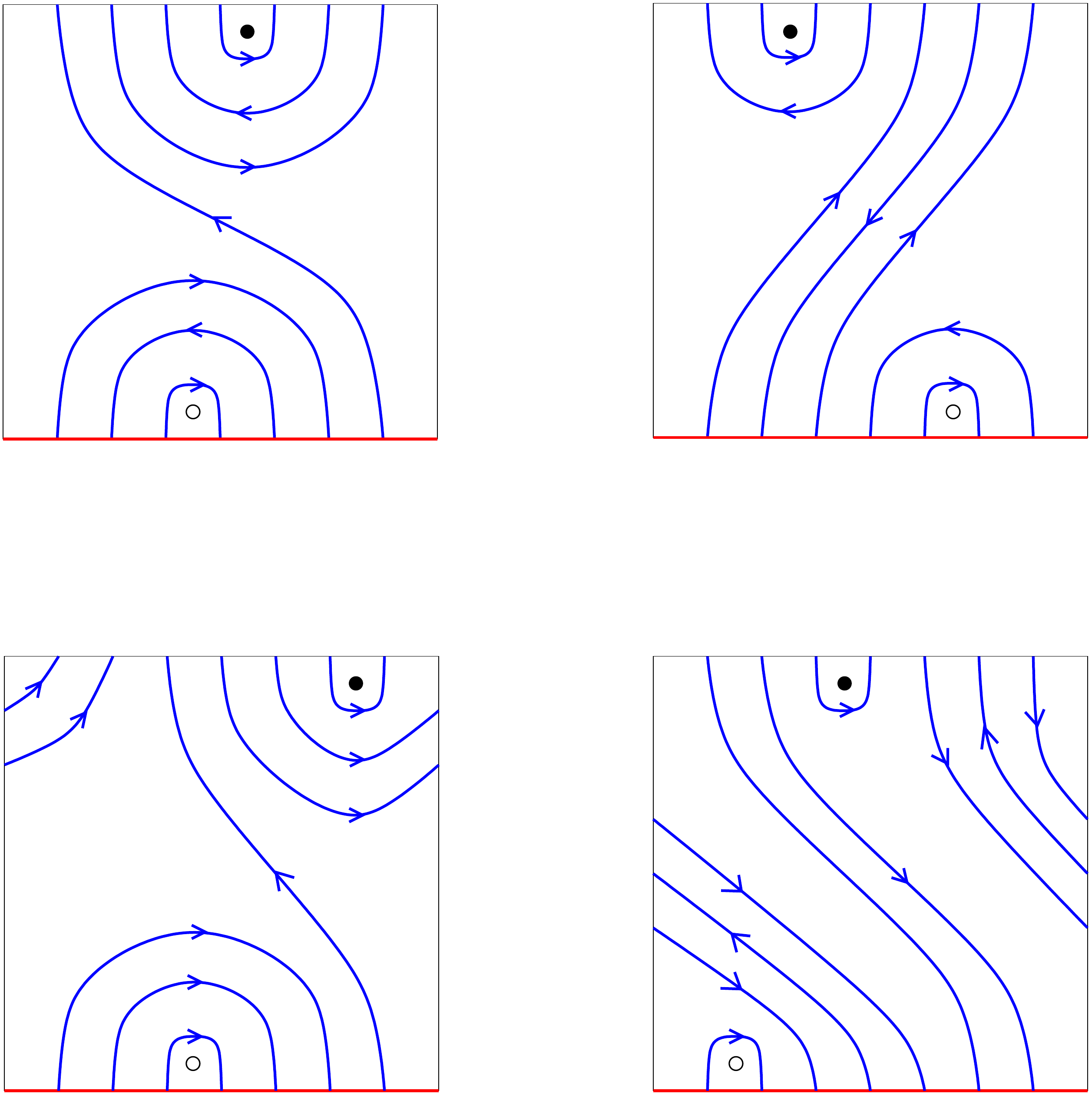}
\put(-255,180){$w$}
\put(-242,280){$z$}
\put(-300,175){\color{red} $\alpha$}
\put(-260,220){\color{blue} $\beta$}
\caption{At top, two reduced diagrams of the twist knot $5_2$, and at bottom, two reduced diagrams of the torus knot $T(2,7)$.  The top two are incoherent: with the $\beta$ curve oriented, the ``rainbow" arcs over the $w$ basepoint do not all orient the same way.  The bottom two are positive coherent, as all of the rainbow arcs do orient the same way, and $\alpha \cdot \beta > 0$.}
\label{fig:examples}
\end{figure}
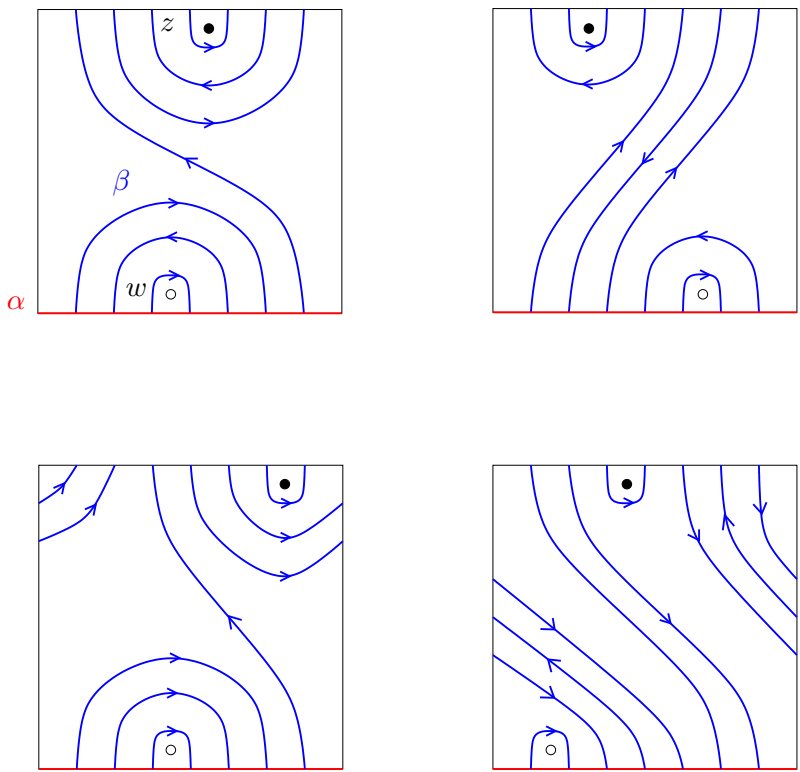

\section{Proof of the characterization.}
\label{sec:Background}

We assume familiarity with (knot) Floer homology and review the essential input for our work in Subsections \ref{ss:HFK} and \ref{ss:diagrams}.
In particular, we follow the treatment of \cite[\S 2.2]{Rasmussen2015floer}, with slight differences in notation.
We prove Theorem \ref{thm:main} in Section \ref{ss:11}.


\subsection{Knot Floer homology and L-space knots.}
\label{ss:HFK}
Let $K$ denote a (doubly pointed) knot in a rational homology sphere $Y$.
Let $N(K)$ denote an open tubular neighborhood of $K$, $X = Y \smallsetminus N(K)$ the knot exterior, and $\mu \subset \del X$ a meridian of $K$.
Let $\Spc(Y)$ denote the set of $\spc$ structures on $Y$ and $\Spc(X,\del X)$ the set of relative $\spc$ structures on $(X,\del X)$.
They are torsors over the groups $H_1(Y;\bZ)$ and $H_1(X;\bZ)$, respectively.
Let $\Spc(K)$ denote the set of orbits in $\Spc(X,\del X)$ under the action by $[\mu]$.
It forms a torsor over $H_1(X;\bZ) / [\mu] \approx H_1(Y;\bZ)$, and there exists a pair of torsor isomorphisms $i_v, i_h : \Spc(K) \to \Spc(Y)$.

We work with the hat-version of knot Floer homology with $\bZ$ coefficients, graded by $s \in \Spc(K)$:
\[
\widehat{HFK}(K) = \bigoplus_{s} \widehat{HFK}(K,s).
\]
There exists a further pair of gradings
\[
a, m : \widehat{HFK}(K,s) \to \bZ
\]
on each summand, the Alexander and Maslov gradings.
An element is homogeneous if it is homogeneous with respect to both gradings.
The group $\widehat{HFK}(K)$ comes equipped with differentials $\widetilde{d}_v$, $\widetilde{d}_h$ that preserve the $\Spc(K)$-grading, respectively lower $m$ and $m-2a$ by one, and respectively raise and lower $a$.
They are invariants of $K$, and their homology calculates $\widehat{HF}(Y,i_v(s))$ and  $\widehat{HF}(Y,i_h(s))$, respectively.
The manifold $Y$ is an L-space if $\widehat{HF}(Y,t) \approx \bZ$ for all $t \in \Spc(Y)$.

\begin{defin}
\label{def:positivechain}
For a knot $K$ in an L-space $Y$ and $s \in \Spc(K)$, the group $\widehat{HFK}(K,s)$ is a {\bf positive chain} if it admits a homogeneous basis $x_1,\dots,x_{2n+1}$ such that, for all $k$,
\[
\widetilde{d}_v(x_{2k}) = \pm x_{2k-1}, \quad \widetilde{d}_h(x_{2k}) = \pm x_{2k+1}, \quad \widetilde{d}_v(x_{2k-1}) = \widetilde{d}_h(x_{2k-1}) = 0.
\]
The group $\widehat{HFK}(K)$ {\em consists of positive chains} if $\widehat{HFK}(K,s)$ is a positive chain for all $s$.
\end{defin}

\noindent
For example, a positive chain with $2n+1 = 7$ generators takes the form shown here:

\[
\xymatrix{
& \langle x_2 \rangle \ar[dl]_{\widetilde{d}_v} \ar[dr]^{\widetilde{d}_h} & & \langle x_4 \rangle\ar[dl]_{\widetilde{d}_v} \ar[dr]^{\widetilde{d}_h} & & \langle x_6 \rangle \ar[dl]_{\widetilde{d}_v} \ar[dr]^{\widetilde{d}_h} & \\
\langle x_1 \rangle & & \langle x_3 \rangle & &\langle x_5 \rangle & & \langle x_7 \rangle
}
\]
\noindent
Each arrow represents a component of the differential and is an isomorphism between the groups it connects.
Similarly, a {\em negative chain} is the dual complex to a positive chain with respect to the defining basis.
Reversing the arrows above gives an example of a negative chain.

Note that the requirement that the positive chain basis elements are homogeneous does not appear in \cite[Definition 3.1]{Rasmussen2015floer}.
For example, we could alter the positive chain basis displayed above to an inhomogeneous one by replacing $x_2$ by $x_2+x_1$.
However, homogeneity is an intended property of a positive chain basis in the literature.
Moreover, an inhomogeneous positive chain basis gives rise to a homogeneous one by replacing each basis element by its homogeneous part of highest bigrading.
Therefore, our definition is no more restrictive, and its precision is more convenient when we invoke it in the proofs of Lemmas \ref{lem:nodifferentials} and \ref{lem:multiplicity}.

Let $\lambda \subset \del X$ denote the rational longitude of $K$, the unique slope that is rationally null-homologous in $X$.
Note that $\mu \ne \lambda$, since $Y$ is a rational homology sphere.
Another slope $\alpha \subset \del X$ is {\em positive} or {\em negative} according to the sign of $(\mu \cdot \lambda)(\lambda \cdot \alpha)(\alpha \cdot \mu)$, for any orientations on these curves.
The knot $K$ is a positive or a negative L-space knot if it has an L-space surgery slope of that sign.
The following theorem characterizes positive L-space knots in L-spaces in terms of their knot Floer homology.
Ozsv\'ath and Szab\'o originally proved it for the case of knots in $S^3$ \cite[Theorem~1.2]{Ath}.
Boileau, Boyer, Cebanu, and Walsh promoted a significant component of their result to knots in rational homology spheres \cite{Boileauetal2012}.
Building on it, J. and S. Rasmussen established the definitive form of the result that we record here: see \cite[Lemmas~3.2, 3.3, 3.5]{Rasmussen2015floer}, including the proofs of these results.

\begin{thm}
\label{thm:PositiveLspace}
A knot $K$ in an L-space $Y$ is a positive L-space knot if and only if $\widehat{HFK}(K)$ consists of positive chains. \qed
\end{thm}

\noindent
Similarly, $K$ is a negative L-space knot if and only if $\widehat{HFK}(K)$ consists of negative chains.
The reason amounts to the behavior of Dehn surgery and knot Floer homology under mirroring.


\subsection{Calculating the invariants from a Heegaard diagram.}
\label{ss:diagrams}

The invariants can be calculated from any doubly-pointed Heegaard diagram $D = (\Sigma,\alpha,\beta,z,w)$ of $K$ after making some additional analytic choices.
Here, as usual, $\Sigma$ denotes a closed, oriented surface of some genus $g$;  $\alpha$ and $\beta$ are $g$-tuples of homologically linearly independent, disjoint, simple closed curves in $\Sigma$; and the two basepoints $w$ and $z$ lie in the complement of the $\alpha$ and $\beta$ curves on $\Sigma$.
The curve collections induce tori $\bT_\alpha, \bT_\beta$ in the $g$-fold symmetric product $\mathrm{Sym}^g(\Sigma)$.
The underlying group of the Floer chain complex $\widehat{CFK}(D)$ is freely generated by $\fT(D) = \bT_\alpha \cap \bT_\beta$.
The elements of $\fT(D)$ fall into equivalence classes in 1-1 correspondence with $\Spc(K)$: two elements $x, y$ lie in the same equivalence class iff the set $\pi_2(x, y)$ of homotopy classes of Whitney disks from $x$ to $y$ is non-empty.
Write $\fT(D,s)$ for the equivalence class corresponding to $s \in \Spc(K)$ and $\widehat{CFK}(D,s)$ for the subgroup of $\widehat{CFK}(D)$ generated by the elements in $\fT(D,s)$.
Each Whitney disk $\phi \in \pi_2(x,y)$ has a pair of multiplicities $n_z(\phi)$, $n_w(\phi)$ and a Maslov index $\mu(\phi)$.
The Alexander and Maslov gradings on $\widehat{CFK}(D,s)$ are characterized up to an overall shift by the relations
\[
a(x) - a(y) = n_z(\phi)-n_w(\phi), \quad m(x) - m(y) = \mu(\phi) - 2 n_w(\phi)
\]
for all $x, y \in \fT(D,s)$ and $\phi \in \pi_2(x,y)$.
There exist endomorphisms $d_0, d_v, d_h$ of $\widehat{CFK}(D,s)$ defined on generators $x \in \fT(D,s)$ by the same general prescription:
\[
d(x) =\sum_{ y , \phi } \# \widehat{\mathcal M}(\phi) \cdot y.
\]
Here $y$ ranges over $\fT(D,s)$; $\phi$ ranges over the elements of $\pi_2(x,y)$ with Maslov index $\mu(\phi)=1$ and a constraint on the multiplicities $n_w(\phi)$, $n_z(\phi)$; and $\# \widehat{\mathcal M}(\phi)$ is the count of pseudo-holomorphic representatives of $\phi$.
The specific constraints on the multiplicities are
$n_w(\phi) = n_z(\phi)=0$ for $d=d_0$;
$n_w(\phi)=0, n_z(\phi) > 0$ for $d=d_v$;
and $n_w(\phi) > 0, n_z(\phi)=0$ for $d=d_h$.
The maps $d_0, d_0+d_v,d_0+d_h$ are all differentials.
The differential $d_0 + d_v$ lowers $m$ by one and is filtered with respect to $a$.
The differential $d_0 + d_h$ lowers $m-2a$ by one and is filtered with respect to $-a$.
The differential $d_0$ is the $a$-filtration-preserving component of each.
The groups $\widehat{HFK}(K,s)$, $\widehat{HF}(Y,i_v(s))$, and $\widehat{HF}(Y,i_h(s))$ are the homology groups of $\widehat{CFK}(D,s)$ with respect to $d_0$, $d_0 + d_v$, and $d_0 + d_h$, respectively, and the Alexander and Maslov gradings on $\widehat{CFK}(D,s)$ descend to the respective gradings on $\widehat{HFK}(K,s)$.
The maps $d_0 + d_v$ and $d_0 + d_h$ induce the differentials $\widetilde{d}_v$ and $\widetilde{d}_h$ on $\widehat{HFK}(K,s)$, respectively.

The Alexander and Maslov gradings enable us to constrain the existence of pseudo-holomorphic disks for L-space knots.

\begin{lem}
\label{lem:nodifferentials}
Suppose that $D$ is a doubly-pointed Heegaard diagram for $K$, $d_0$ vanishes on $\widehat{CFK}(D,s)$, and $\widehat{HFK}(K,s)$ is a positive chain with basis $x_1,\dots,x_{2n+1}$.
If there exist generators $x,y \in \fT(D,s)$ and a disk $\phi \in \pi_2(x,y)$ with $\mu(\phi)=1$ and $\# \cM(\phi) \ne 0$, then $x = x_{2k}$ and $y = x_{2k\pm1}$ for some $k$.
Furthermore, $n_z(\phi) = 0$ if $j = 2k+1$ and $n_w(\phi) = 0$ if $j = 2k-1$.
\end{lem}

\begin{proof}
Definition \ref{def:positivechain} and the paragraph preceding it show that the positive chain basis satisfies
\[
a(x_{2k})-a(x_{2k-1}) = b_k, \quad m(x_{2k}) - m(x_{2k-1}) = 1
\]
and
\[
a(x_{2k})-a(x_{2k+1}) = -c_k, \quad m(x_{2k}) - m(x_{2k+1}) = 1 - 2c_k,
\]
for some positive integers $b_k, c_k$, $k=1,\dots,n$.
In particular, the Maslov and Alexander gradings of the $x_l$ both decrease with the index $l$.
The generators $x,y \in \fT(D,s)$ are homogeneous with respect to both gradings, as are the positive chain basis elements, so it follows that $x = \pm x_i$ and $y = \pm x_j$ for some $i$, $j$, and choices of sign.

The assumption that $\# \cM(\phi) \ne 0$ implies that $n_w(\phi), n_z(\phi) \ge 0$, and at least one inequality is strict, since $d_0 = 0$.
Thus,
\begin{equation}
\label{eq:gradings1}
a(x_i) - a(x_j) = n_z(\phi) - n_w(\phi) \quad \textup{and} \quad m(x_i) - m(x_j) = 1 -2n_w(\phi) \le 1.
\end{equation}
In the case of equality $m(x_i)-m(x_j)=1$, then $j = i-1$, and either $i = 2k$, which gives the desired conclusion, or else $j = 2k$, which we must rule out.
In this case, $n_w(\phi) = 0$ and $n_z(\phi) > 0$.
Since $Y$ is a rational homology sphere, it follows that $\phi \in \pi_2(x,y)$ is the unique disk with $\mu(\phi)=1$, so the coefficient on $y$ in $\widetilde{d}_v(x)$ is $\# \cM(\phi) \ne 0$.
However, this implies that $\widetilde{d}_v(x_{2k+1}) \ne 0$, which violates the form of the positive chain.

Otherwise, $m(x_i)-m(x_j) \le -1$, so $i < j$.
We have
\begin{equation}
\label{eq:gradings2}
a(x_i) - a(x_j) = \sum_{l=i}^{j-1} a(x_l) - a(x_{l+1}), \quad m(x_i) - m(x_j) = \sum_{l=i}^{j-1} m(x_l) - m(x_{l+1}).
\end{equation}
Each term in the second sum is odd, and their total sum is odd, so it contains an odd number of terms.
The terms of the first sum alternately take the form $-b_k$ and $-c_k$, while the terms of the second sum alternately take the form $-1$ and $1-2c_k$.
Let $b$ denote the sum of the values $b_k$ that appear and $c$ the sum of the $c_k$ that appear.
All of values $b_k$ are positive, so $b \ge 0$, and $b =0$ only if $i$ is even and $j=i+1$, and $b=1$ only if $i$ is odd and $j=i+1$.
The first sum in \eqref{eq:gradings2} equals $-b-c$, while the second sum equals $1-2c$ or $-1-2c$, according to whether $i$ is even or odd.
Comparing with the second sum in \eqref{eq:gradings1} gives $n_w(\phi) = c$ or $c+1$, according to whether $i$ is even or odd.
Comparing with the first sum in \eqref{eq:gradings1} then gives $n_z(\phi) = -b$ or $-b+1$, according to whether $i$ is even or odd.
If $i$ is even, then $0 \le n_z(\phi) = -b$ gives $n_z(\phi) = 0$ and $j=i+1$, which gives the desired conclusion.
If $i$ is odd, then $0 \le n_z(\phi) = -b+1$ again gives $n_z(\phi)=0$ and $j=i+1$.
However, in this case it follows as before that $y$ appears with coefficient $\# \cM(\phi) \ne 0$ in $\widetilde{d}_h(x)$, which once again violates the form of the positive chain.
\end{proof}


\subsection{(1,1) knots.}
\label{ss:11}

Now assume that  $D$ is a $(1,1)$-diagram, so $\Sigma$ has genus one.
In this case, the differentials on $\widehat{CFK}(D)$ admit an explicit description that requires no analytic input, owing to the Riemann mapping theorem and the correspondence between holomorphic disks in the Riemann surface $\Sigma = \mathrm{Sym}^1(\Sigma)$ and its universal cover $\bC$.
Specifically, the conditions that $\phi \in \pi_2(x,y)$ and $\mu(\phi)=1$ imply that $\# \widehat{\mathcal M}(\phi) = \pm 1$ for any choice of analytic data.
Furthermore, these conditions are met if and only if the image of $\phi$ lifts under the universal covering map $\pi : \bR^2 \to \Sigma$
to a bigon cobounded by lifts of $\alpha$ and $\beta$.
See \cite{Goda2005} and \cite[pp.89-96]{Ozsvath2004} for more details.

\begin{figure}
\includegraphics[width=2.5in]{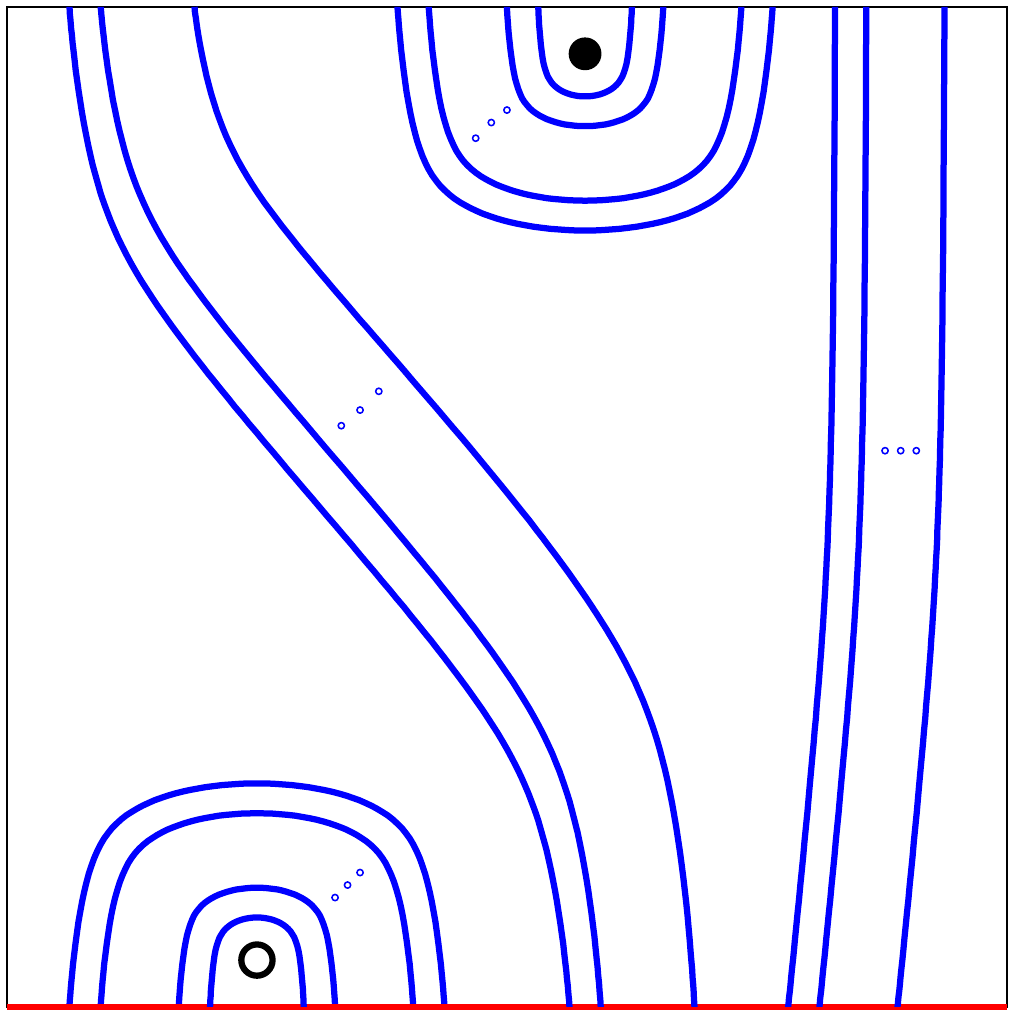}
\put(-100,168){$z$}
\put(-158,8){$w$}
\caption{Standard form for a reduced $(1,1)$ diagram, taken from \cite[Figure 2]{Rasmussen2005}.  The left and right sides of the square are identified in the standard way and the top and bottom sides by a twist.}
\label{fig:pqrs}
\end{figure}

Assuming now that $D$ is reduced, we have $d_0 = 0$ and $\widehat{HFK}(K,s) \approx \widehat{CFK}(D,s)$.
We proceed to analyze the differentials $\widetilde{d}_h$ and $\widetilde{d}_v$.
Without loss of generality, we henceforth fix $\Sigma = \bR^2 / \bZ^2$, $z = \pi(\bZ^2)$, a horizontal line $\wa \subset \bR^2 \smallsetminus \bZ^2$, and $\alpha = \pi(\wa)$.
After a further homeomorphism, we may assume that $D$ takes the form shown in Figure \ref{fig:pqrs}.
Observe that the embedded bigons in $D$ are the obvious ones cobounded by the $\alpha$ with the ``rainbow" arcs of $\beta$ above $w$ and below $z$.
Coherence is the condition that when the $\beta$ curve is oriented, all of the rainbow arcs around a fixed basepoint orient the same way.
This makes it easy to check coherence from a diagram in standard form.

Let $H^\pm$ denote the upper and lower half-planes bounded by $\wa$.
Choose a lift to $\wa$ of any point in $\fT(D,s)$.
There exists a unique lift $\wb_s \subset \bR^2$ that passes through it.
The points of $\wa \cap \wb_s$ are in 1-1 correspondence with $\fT(D,s)$ under $\pi$, and there are an odd number $2n+1$ of them, since $\wa \cdot \wb_s = \chi(\widehat{CFK}(D,s))=1$ with appropriate orientations on the curves.
The curve $\wb_s$ meets each half-plane in one closed ray and $n$ compact arcs.
These compact arcs cobound {\em positive bigons} $D_k^+ \subset H^+$ and {\em negative bigons} $D_k^- \subset H^-$ with $\wa$ for $k=1,\dots,n$.
A positive bigon attains a local maximum, and its image under $\pi$ is the top of one of the rainbow arcs above $w$.
Consequently, it contains a lift of the bigon in $D$ cobounded by that rainbow arc with $\alpha$.
Thus, $n_w(D_k^+) > 0$, and similarly $n_z(D_k^-) > 0$, for all $k$.
Lastly, the positive bigons go from their left-most corner to their right-most corner, and vice versa for the negative bigons.

\begin{lem}
\label{lem:multiplicity}
If $D$ is a reduced $(1,1)$ diagram of $K$ and $\widehat{HFK}(K,s)$ is a positive chain, then
\begin{enumerate}
\item
$n_z(D_k^+) = 0$ and $n_w(D_k^-) = 0$ for all $k$;
\item
the positive bigons are exactly the ones that contribute to $\widetilde{d}_v$, the negative bigons are exactly the ones that contribute to $\widetilde{d}_h$, and there are no other bigons; and
\item
if $\alpha$ and $\beta$ are oriented so that $\alpha \cdot \beta > 0$, then $\wa$ and $\wb_s$ induce coherent orientations on the boundaries of all of the bigons they cobound.
\end{enumerate}
\end{lem}

\begin{proof}
\begin{inparaenum}
\item
The conditions of Lemma \ref{lem:nodifferentials} are met by each positive and negative bigon and its corners, since $D$ is a reduced $(1,1)$ diagram and each positive and negative bigon is an embedded bigon in $\bR^2$.
The conclusion now follows from the last conclusion of Lemma \ref{lem:nodifferentials} and the remark about multiplicities just before this Lemma.

\item
The positive chain basis elements are homogeneous and in different bigradings, while generators are homogeneous.
It follows that the positive chain basis is comprised of generators (up to sign).
Suppose that there exists an embedded bigon with $n_w = 0$ that goes from a lift of a generator $x$ to a lift of a generator $y$ in $\wa \cap \wb_s$.
Then it is the unique such bigon, and the coefficient of $y$ in $\widetilde{d}_v(x)$ is non-zero.
Moreover, $\widetilde{d}_v(x) = \pm y$, because $x$ and $y$ are (signed) positive chain basis elements.
By part (1), each of the $n$ positive bigons contributes to $\widetilde{d}_v$ and connects a different pair of generators.
Since $\rk \, \widehat{HFK}(K,s) = 2n+1$, positive bigons are precisely those that contribute to $\widetilde{d}_v$.
The same remarks apply to the negative bigons and $\widetilde{d}_h$.
Lemma \ref{lem:nodifferentials} shows that there are no other bigons.

\item
By part (2) and the remark just before the Lemma, it follows that the points of intersection between $\wa$ and $\wb_s$ occur in the same order along these curves, up to reversal.
Suppose that $\alpha \cdot \beta > 0$, so $\wa \cdot \wb_s = +1$.
Without loss of generality, $\wa$ orients from left to right.
It suffices to check coherence on the boundary of the bigon with a corner at the leftmost point of intersection $x \in \wa \cap \wb_s$.
Note that $x$ is a positive point of intersection, as the signs of the intersections alternate along $\wa$ and sum to $+1$.
If there is no such bigon, then the conclusion is automatic (and $\mathrm{rk} \, \widehat{HFK}(K,s) = 1$ in this case).
If there is a bigon, label its other corner $y$.
If it is a positive bigon, then $\widetilde{d}_h(x) = y$, while $\widetilde{d}_v(x)=0$, in violation of the form of $\widehat{HFK}(K,s)$.
Therefore, it is a negative bigon, and its boundary is coherently oriented by $\wa$ and $\wb_s$.
\end{inparaenum}
\end{proof}

Drawing on the proof of Lemma \ref{lem:multiplicity}, we make a definition and prove a converse.

\begin{defin}
\label{def:graphic}
A curve $\wb \subset \bR^2$ is {\bf graphic} if its intersection points with $\widetilde{\alpha}$ occur in the same or opposite orders along $\widetilde{\alpha}$ and $\widetilde{\beta}$.
It is {\em positive} if they occur in opposite orders and {\em negative} if they occur in the same order.
\end{defin}

\noindent
The bottom right picture in Figure \ref{fig:HFK} displays a positive graphic curve.
The reason for the terminology is that a curve $\wb \subset \bR^2$ is graphic iff there exists a homeomorphism of $\bR^2$ taking $(\wa,\wb)$ to the pair consisting of the $x$-axis and the graph of an odd degree polynomial with all roots real and distinct.
Its sign, positive or negative, is minus the sign of the leading coefficient of the polynomial.

\begin{prop}
\label{prop:graphic}
If $D$ is a reduced $(1,1)$ diagram of $K$, then $\widehat{HFK}(K,s)$ is a positive or a negative chain if and only if $\wb_s$ is positive or negative graphic, respectively.
\end{prop}

\begin{proof}
If $\widehat{HFK}(K,s)$ is a positive chain, then Lemma \ref{lem:multiplicity} parts (2) and (3) show that $\wb_s$ is positive graphic.
Conversely, suppose that $\wb_s$ is positive graphic.
Label the points of intersection in $\wa \cap \wb_s$ by $x_1,\dots,x_{2n+1}$ in the order they occur along $\wa$.
The only bigons cobounded by $\wa$ and $\wb_s$ are the positive and negative bigons. This latter fact can be verified using the equivalent definition of a graphic curve mentioned after Definition~\ref{def:graphic}. 
Therefore,
\[
\widetilde{d}_v(x_{2k}) = \delta_k \cdot x_{2k-1}, \quad \widetilde{d}_h(x_{2k}) = \epsilon_k \cdot x_{2k+1}, \quad \widetilde{d}_v(x_{2k-1}) = \widetilde{d}_h(x_{2k-1}) = 0,
\]
for all $k$ and some $\delta_k, \epsilon_k \in \{-1,0,+1\}$.  A value $\delta_k$ or $\epsilon_k$ is 0 iff the corresponding bigon contains both $z$ and $w$ basepoints.
Since
\[
H_*(\widehat{HFK}(K,s),\widetilde{d}_v) \approx \widehat{HF}(Y,i_v(s)) \approx \bZ
\quad \textup{and} \quad
H_*(\widehat{HFK}(K,s),\widetilde{d}_h) \approx \widehat{HF}(Y,i_h(s)) \approx \bZ,
\]
it follows that $\delta_k, \epsilon_k \ne 0$ for all $k$, and $\widehat{HFK}(K,s)$ is a positive chain.
The corresponding statements for a negative chain and negative coherent diagram follow as well.
\end{proof}

\begin{proof}
[Proof of Theorem \ref{thm:main}]
Let $D$ denote a reduced $(1,1)$-diagram of $K$.

Suppose first that $K$ is a positive L-space knot.
Orient $\alpha$ and $\beta$ so that $\alpha \cdot \beta > 0$, and suppose that $B$ is an embedded bigon in $D$.
We must show that $\alpha$ and $\beta$ coherently orient $\del B$.
Suppose that the corners of $B$ belong to $\fT(s)$, and let $\widetilde{B}$ denote the lift of $B$ to $\bR^2$ with corners in $\wa \cap \wb_s$.
We have $\wa \cdot \wb_s = +1$.
Since $K$ is a positive L-space knot, Theorem \ref{thm:PositiveLspace} implies that $\widehat{HFK}(K,s)$ is a positive chain, and by Lemma \ref{lem:multiplicity}(3), the orientations on $\wa$ and $\wb_s$ coherently orient the boundaries of all of the bigons they cobound.
In particular, they coherently orient $\del \widetilde{B}$, so $\alpha$ and $\beta$ coherently orient $\del B$.
Therefore, $D$ is positive coherent.

Conversely, suppose that $D$ is positive coherent.
Orient $\alpha$ and $\beta$ so that $\alpha \cdot \beta > 0$ and $\wa$ orients from left to right.
Fix $s \in \Spc(K)$ and select any $x \in \wa \cap \wb_s$.
It is the endpoint of a closed ray of $\wb_s$ oriented out of it.
If this ray meets $\wa$ in another point, then let $y$ denote the first such point.
The arc of $\wb_s$ from $x$ to $y$ cobounds a bigon $\widetilde{B}$ with $\wa$.
If $y$ lies to the left of $x$, then $\del \widetilde{B}$ is incoherently oriented by $\wa$ and $\wb_s$.
The arc of $\wb_s$ has an extremal (highest or lowest) point which lifts the extremal point of a rainbow arc.
Projecting $\widetilde{B}$ to $D$ by $\pi$, it follows that this rainbow arc cobounds an embedded bigon with $\alpha$ that is incoherently oriented by $\alpha$ and $\beta$, a contradiction.
Therefore, $y$ lies to the right of $x$.
It follows by induction that the points of $\wa \cap \wb_s$ occur in the opposite orders along $\wa$ and $\wb_s$ with respect to these curves' orientations.
It follows that $\wb_s$ is graphic, and since $\alpha \cdot \beta > 0$, it is positive graphic.
By Proposition \ref{prop:graphic}, $\widehat{HFK}(K,s)$ is a positive chain.
Therefore, $\widehat{HFK}(K)$ consists of positive chains, and $K$ is a positive L-space knot by Theorem \ref{thm:PositiveLspace}.

The corresponding statements for negative coherent diagrams and negative L-space knots follow by a similar argument.
\end{proof}


\section{1-bridge braids.}
\label{sec:1-bridge}

This section consists of two parts.
The first is devoted to the proof of Theorem~\ref{thm:1-Bridge} and its sign-refined version, Theorem~\ref{thm:1-bridgerefinement}.
The second is a vignette on how Theorem~\ref{thm:1-bridgerefinement} pertains to solid torus fillings on 1-bridge braid exteriors.
Specifically, we give a novel argument in Proposition \ref{prop:unique} to characterize the {\em Berge manifold} amongst 1-bridge braid exteriors.


\subsection{1-bridge braids are L-space knots}
\label{ss:1-bridge}
We prepare by describing a reduced $(1,1)$-diagram of a 1-bridge braid in a rational homology sphere.

Let $\gamma \cup \delta \subset S^1 \times D^2$ denote a 1-bridge braid.
Identify $\del (S^1 \times D^2)$ with the flat torus $\Sigma$ in such a way that meridians $\theta \times \del D^2$ lift to horizontal lines and longitudes $S^1 \times x$ lift to vertical lines under $\pi : \bR^2 \to \Sigma$.
Isotope $\gamma$ rel endpoints into a geodesic on $\Sigma$, so that $\pi^{-1}(\gamma)$ consists of parallel line segments in $\bR^2$ of some common slope $s(\gamma) \in P^1(\bR)$.
Orient $\gamma$ and write $\del \gamma = z - w$ so that in each lift of $\gamma$ to $\bR^2$, the lift of $z$ lies below that of $w$.

Form a genus one rational homology sphere $Y$ by gluing a solid torus $V$ to $S^1 \times D^2$ along their boundaries, and let $K$ denote the image of $\gamma \cup \delta$ in $Y$.
Let $\alpha$ denote a meridian of $S^1 \times D^2$ and let $\beta_0$ denote the curve to which the meridian of $V$ attaches, which we take to be a geodesic.
In meridian-longitude coordinates, $\beta_0$ has some slope $p/q \ne 0/1$, and $Y$ is homeomorphic to $L(p,q)$ if $p \ne 1$ and $S^3$ otherwise.
Assume that $\alpha$ and $\beta_0$ are disjoint from $\del \gamma$.
Isotope $\beta_0$ by a finger-move along $\gamma$ into a curve $\beta_1$ disjoint from $\gamma$.
Observe that $(\Sigma,\alpha,\beta_1,z,w)$ is a doubly-pointed Heegaard diagram for $K \subset Y$.
Put it into reduced form by further isotoping $\beta_1$ so as to eliminate all bigons that do not contain basepoints, one by one, resulting in a curve $\beta_2$.
Let $D = (\Sigma,\alpha,\beta_2,z,w)$ denote the resulting reduced diagram of $K \subset Y$.
Figure \ref{fig:HFK} exhibits this procedure for the inclusion of the 1-bridge braid $K(7,4,2)$ into $S^3$, the pretzel knot $P(-2,3,7)$.
The notation derives from the braid presentation of 1-bridge braids; see \cite{Berge1991, Gabai1990,wu2004}. The $1$--bridge braid $K(\omega,b,m)$ denotes the closure of the braid word $(\sigma_b \sigma_{b-1} \cdots \sigma_1)(\sigma_{\omega-1} \sigma_{\omega-2} \cdots \sigma_1)^m$ in $S^1 \times D^2$, where the $\sigma_i$ are the generators of the braid group~$B_w$. Note that $K(\omega, b, m)$ and $K(\omega, b, m+\omega)$ differ by a full twist, so there is a diffeomorphism of the solid torus that carries one to the other. Therefore, we may assume that $0\le m <w$.

\begin{figure}
\includegraphics[width=6in]{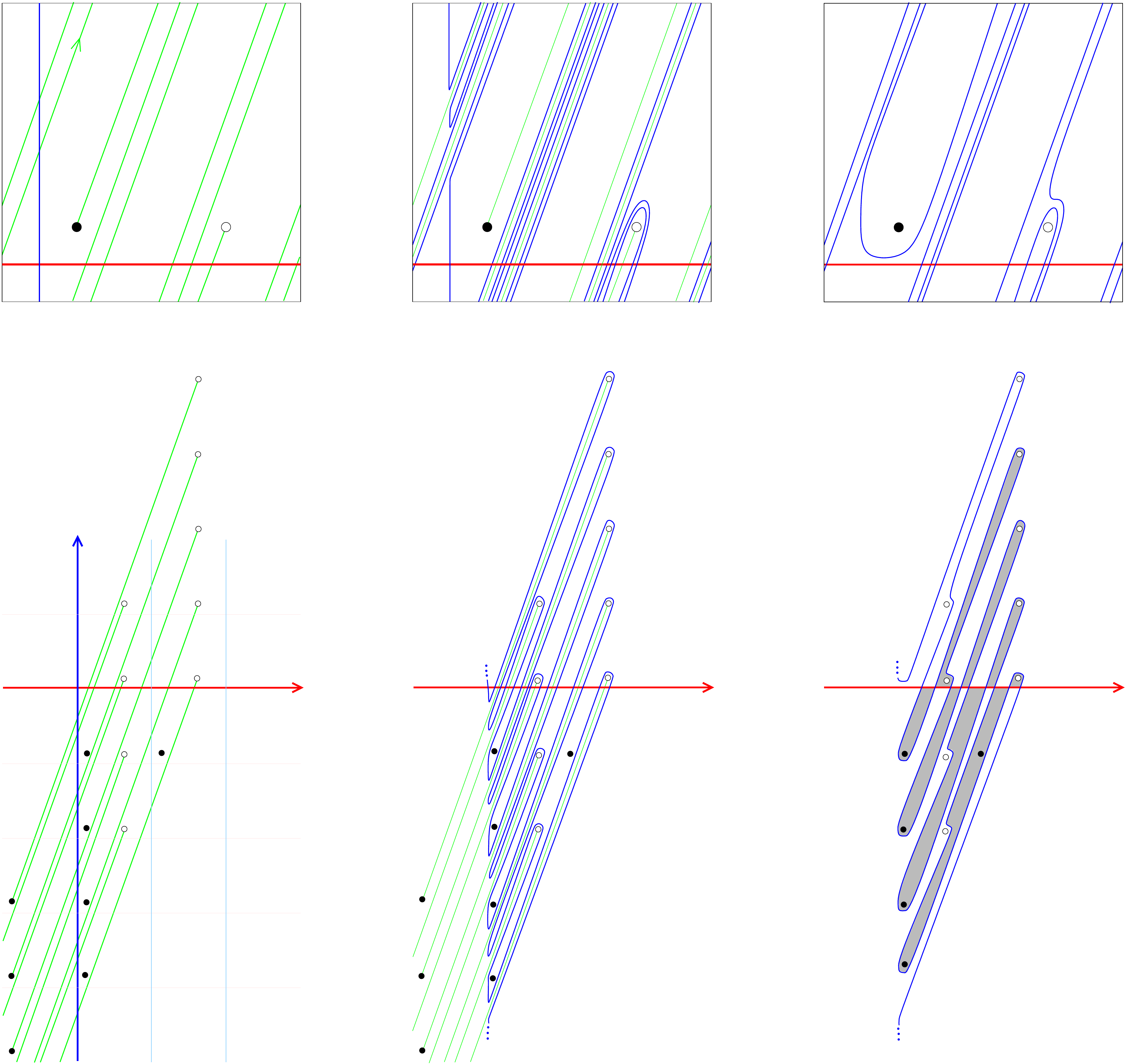}
\put(-415,315){$z$}
\put(-340,315){$w$}
\put(-445,307){\color{red} $\alpha$}
\put(-400,360){$\gamma$}
\put(-422,280){\color{blue} $\beta_0$}
\put(-262,280){\color{blue} $\beta_1$}
\put(-90,280){\color{blue} $\beta_2$}
\caption{Constructing a positive coherent diagram of $P(-2,3,7)$, the inclusion of $K(7,4,2)$ into $S^3$.
Moving from left to right illustrates the procedure described in the proof of Theorem \ref{thm:main}.
At top are diagrams on $\Sigma$ and at bottom their lifts to $\bR^2$.
The bigons at bottom right are shaded for emphasis.
}\label{fig:HFK}
\end{figure}
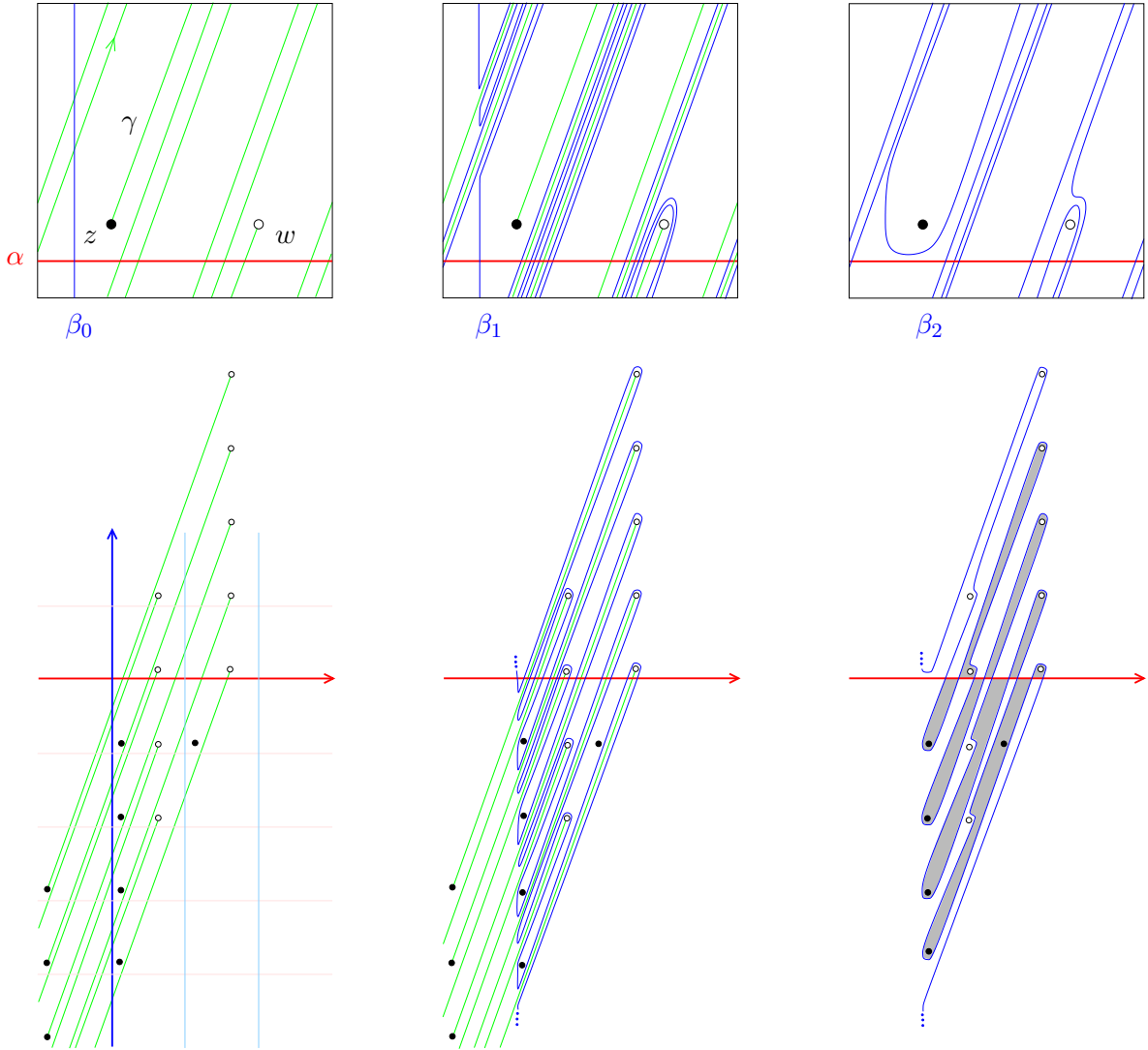

A pair of distinct slopes $s_1, s_2 \in P^1(\bQ)$ determines an interval $[s_1,s_2] \subset P^1(\bQ)$ oriented from $s_1$ to $s_2$ in the counterclockwise sense; thus, $[s_1,s_2] \cup [s_2, s_1] = P^1(\bQ)$ and $[s_1,s_2] \cap [s_2,s_1] = \{s_1,s_2\}$.

\begin{prop}
\label{prop:coherent}
The diagram $D$ is coherent.  Furthermore, it is positive if $s(\gamma) \in [0,p/q]$ and negative if $s(\gamma) \in [p/q,0]$.
\end{prop}

\noindent
Note that Proposition \ref{prop:coherent} does {\em not} characterize the sign of the coherence in terms of $s(\gamma)$.
Theorem \ref{thm:1-bridgerefinement} does so in terms of the {\em slope interval}, and its proof relies on Proposition \ref{prop:coherent}.

\begin{proof}
Choose lifts of $\alpha$ and $\beta_0$ to $\bR^2$.
They meet in a single point of intersection $x_1$.
Consider the lifts of $\gamma$ that meet both $\wa$ and $\wb_0$ and in that order relative to their lifted orientations.
Observe that these lifts meet the same component of $\wa - x_1$.
Label the points of intersection between these lifts and $\wa$ by $y_2,\dots,y_n$ in the order that they appear along this component, moving away from $x_1$.
The isotopy from $\beta_0$ to $\beta_1$ lifts to one from $\wb_0$ to $\wb_1$.
Nearby each $y_i$ is a pair of intersection points $x_{2i}, x_{2i+1}$ between $\wa$ and $\wb_1$.
The points $x_1,\dots,x_{2n+1}$ occur in that order along both $\wa$ and $\wb_1$.
Thus, $\wb_1$ is graphic.
The isotopy from $\beta_1$ to $\beta_2$ lifts to one from $\wb_1$ to $\wb_2$.
Each elimination of a bigon eliminates a pair of intersection points that are consecutive along both $\wa$ and $\wb_1$.
Therefore, the intersection points between $\wa$ and $\wb_2$ are a subset of those between $\wa$ and $\wb_1$, and they occur in the same order along $\wa$ and $\wb_2$.
It follows that $\wb_2$ is graphic, as well.
Furthermore, $\wb_2$ is positive graphic if the $x$-coordinates of $x_1,\dots,x_{2n+1}$ decrease with index and negative graphic if they increase with index.
This is the case if $s(\gamma)$ belongs to $[0,p/q]$ or in $[p/q,0]$, respectively.
In particular, it is independent of the choice of lift of $\beta_2$.
Thus, all lifts of $\beta_2$ are $\varepsilon$-graphic for the same choice of sign $\varepsilon$.
It follows that $D$ is coherent, and the sign-refined statement follows as well.
\end{proof}

\begin{proof}[Proof of Theorem \ref{thm:1-Bridge}]
Immediate from Proposition \ref{prop:coherent} and Theorem \ref{thm:main}.
\end{proof}

\begin{figure}
\includegraphics[width=4in]{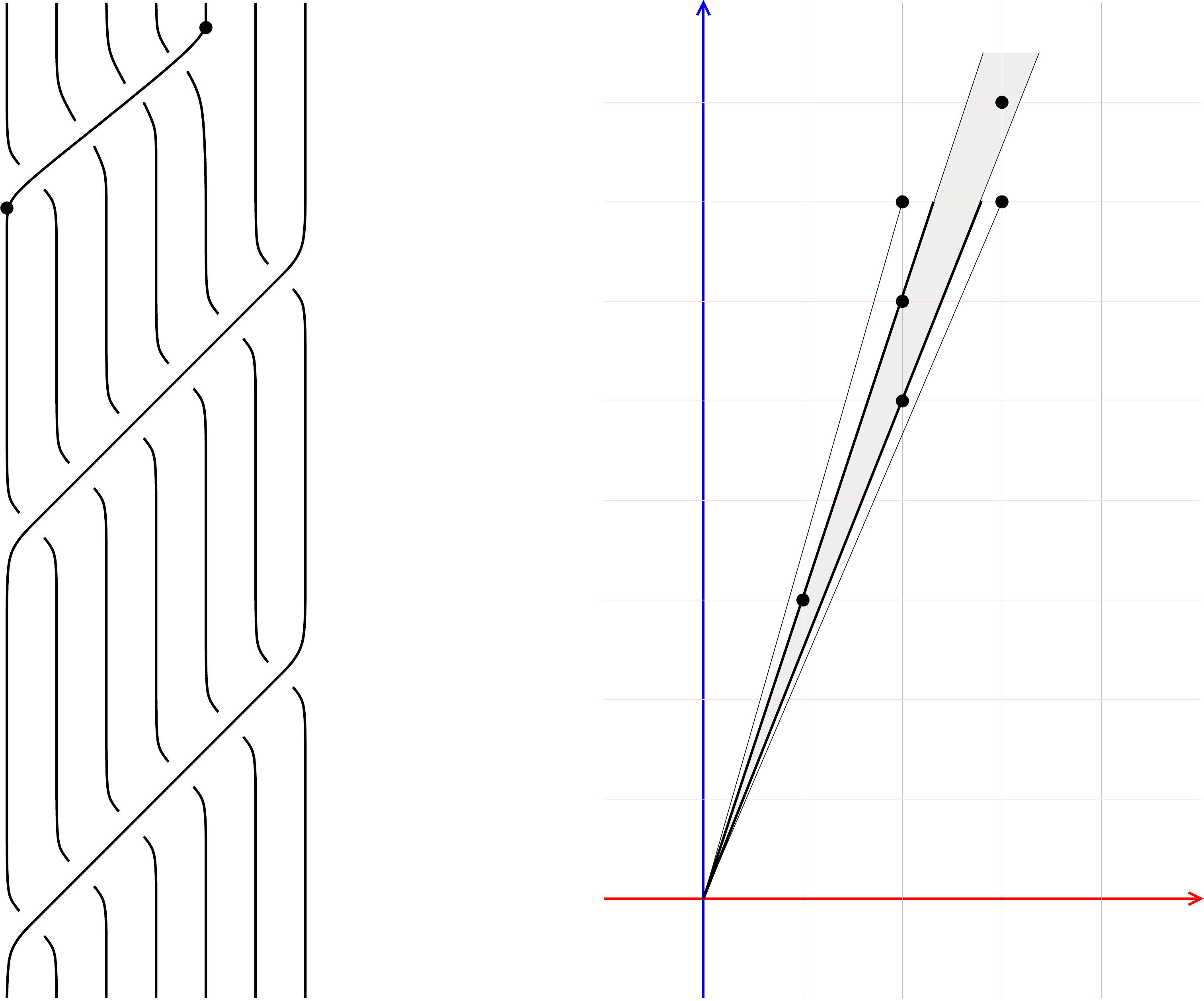}
\caption{At left, a braid with closure the 1-bridge braid $K(7,4,2) \subset S^1 \times D^2$. The bold points suggest its decomposition into the form $\gamma \cup \delta$. At right, a sweep-out of arcs in $\bR^2$ indicating that $I(\gamma) = [5/2,3]$.  The bold points arise in the proof of Proposition \ref{prop:unique}.}
\label{fig:724plane}
\end{figure}

We proceed to sharpen the statement of Theorem \ref{thm:1-Bridge}.
To do so, we introduce the notion of a strict 1-bridge braid and its basic invariants: the {\em winding number} and {\em slope interval}.
Choose the lift of $\gamma$ with one endpoint at $(0,0)$.
Its other endpoint takes the form $(t,\omega) \in \bR \times \bZ$.
We have $\omega > 0$ by our convention on the basepoints, and $t$ is not a proper divisor of $\omega$.
Moreover, $\omega$ is an invariant of the isotopy type of $K$, since $[K] = \omega \cdot [S^1 \times x] \in H_1(S^1 \times D^2; \bZ)$.
This is the {\em winding number} $\omega(K)$.

Consider the sweep-out of line segments with one endpoint at $(0,0)$ and the other endpoint varying along line $y = \omega$.
There exists a maximal {\em slope interval} $I(\gamma) \subset P^1(\bQ)$ containing $s(\gamma)$ and so that the sweep-out of line segments through slopes in $I(\gamma)$ contains no lattice point in its interior.  If $\omega=1$, then $I(\gamma)=P^1(\bQ) \smallsetminus \{0\}$, and otherwise $I(\gamma)= [s_-(\gamma),s_+(\gamma)]$ with $s_-(\gamma) \le s_+(\gamma)$.
The sweep-out descends to an isotopy through 1-bridge braids with slopes in the interior of $I(\gamma)$.
If a line segment of slope $s_\pm(\gamma)$ has endpoint $(q,\omega) \in \bZ^2$, then write $(q,\omega) = (d q', d \omega')$, where $d = \gcd(q,\omega)$.
Let $\mu$ denote the meridian of the torus knot $T(q',\omega')$ and $\lambda$ the surface framing, oriented so that $\mu \cdot \lambda = +1$.
If $d=1$, then $K$ is isotopic to the torus knot $T(q,\omega)$, and if $d > 1$, then $K$ is isotopic to the $(d \lambda \pm \mu)$-cable of $T(q',\omega')$, where the sign is positive or negative if the slope of the line segment is $s_+(\gamma)$ or $s_-(\gamma)$, respectively.
We call such a cable an {\em exceptional cable of a torus knot}.
Otherwise, if no line segment has endpoint $(q,\omega) \in \bZ^2$, then we call $K$ a {\em strict 1-bridge braid} (and justify the terminology in Corollary \ref{cor:invariants}).
In this case, there exists a unique $m \in \bZ$ so that $m < t < m+1$, and $s_-(\gamma)$ and $s_+(\gamma)$ are consecutive terms in the {\em Farey sequence} of fractions from $\omega/(m+1)$ to $\omega/m$ whose denominators are bounded by $|m|$ when expressed in lowest terms.

For the following result, assume $p > q \ge 0$ are coprime integers, and let $L(1,0)$ denote $S^3$.
A {\em simple knot} in $L(p,q)$ is a $(1,1)$ knot whose defining arcs are contained in meridian disks 
in the respective Heegaard solid tori, where the disks' boundaries meet in $p$ transverse points of intersection.
We permit the degenerate case that the the cable arc $\gamma$ is a simple closed curve and the bridge $\delta$ is a point, in which case the simple knot is an unknot.
Equivalently, a simple knot is a knot admitting a $(1,1)$ diagram in which all points of intersection between $\alpha$ and $\beta$ have the same sign.
Note that every simple knot is a 1-bridge braid: to see this, push one of the defining arcs onto the Heegaard torus.

\begin{thm}
\label{thm:1-bridgerefinement}
The inclusion of a strict 1-bridge braid $\gamma \cup \delta \subset S^1 \times D^2$ into $L(p,q)$ is
\begin{enumerate}
\item
a positive L-space knot iff $s_-(\gamma) \in [0,p/q]$;
\item
a negative L-space knot iff $s_+(\gamma) \in [p/q,0]$; and
\item
a simple knot iff $p/q \in I(\gamma)$.
\end{enumerate}
\end{thm}

\noindent
For example, the inclusion of $K(7,4,2)$ into $L(p,q)$ is a simple knot iff $p/q \in [5/2,3]$.

\begin{proof}[Proof of Theorem \ref{thm:1-bridgerefinement}]
The reverse directions of (1) and (2) follow from Proposition \ref{prop:coherent} and Theorem \ref{thm:main}.

For the forward directions, suppose first that $s_+(\gamma) \in (0,p/q)$.
Write $s_+(\gamma) = r/s$ in lowest terms, $s > 0$.
As before, let $\wg$ denote the lift of $\gamma$ with one endpoint at $(0,0)$ and the other at $(t,\omega) \in \bR \times \bZ$.
Note that $\omega > |r|$.
Let $\wa \subset \bR^2 \smallsetminus \bZ^2$ denote a horizontal line that separates $(s,r)$ from $(t,\omega)$, and let $\wb_0 \subset \bR^2 \smallsetminus \bZ^2$ denote a line of slope $p/q$ that separates $(s,r)$ from $(0,0)$.
Then $\wa$, $\wb_0$, and $\wg$ intersect in pairs, and since $r/s \in (0,p/q)$, they cobound a triangle containing $(s,r)$ in its interior.
See Figure \ref{fig:latticepoint}(a).

Follow the procedure for producing a coherent diagram $D$ for $K \subset Y$ using the curves $\alpha = \pi(\wa), \beta_0 = \pi(\wb_0) \subset \Sigma$.
The isotopy from $\wb_0$ to $\wb_1$ captures $(s,r)$ in a negative bigon and $(t,\omega)$ in a positive bigon in $(\bR^2, \wa \cup \wb_1)$.
See Figure \ref{fig:latticepoint}(b).
These points remain in bigons of the respective types following the isotopy from $\wb_1$ to $\wb_2$.
It follows that $\widehat{HFK}(K,s)$ is a positive chain of rank greater than one for the $\spc$ structure $s$ corresponding to the pair of lifts $\wa,\wb_2$.
Thus, $K$ is not a negative L-space knot, which gives the forward direction of (2).
The forward direction of (1) follows the same line of reasoning.

Since a simple knot is both a positive and a negative L-space knot, the forward direction of (3) follows.
Finally, taking a 1-bridge presentation of $K$ in which $\gamma$ has slope $p/q$ results in a diagram $D$ of a simple knot, and the reverse direction of (3) follows.
\end{proof}

The following result recovers the isotopy classification of 1-bridge braids from Theorem \ref{thm:1-bridgerefinement}.  Compare \cite[Proposition 2.3]{Gabai1990}, which does so in terms of braid parameters.

\begin{cor}
\label{cor:invariants}
If $K = \gamma \cup \delta \subset S^1 \times D^2$ is a strict 1-bridge braid, then the slope interval $I(\gamma)$ is an invariant of the isotopy type of $K$.
Its isotopy type is determined by its slope interval and winding number, and $K$ is not isotopic to a torus knot or an exceptional cable thereof.
\end{cor}

\begin{proof}
The slope interval is characterized by Theorem \ref{thm:1-bridgerefinement} as the set of surgery slopes for which $K$ includes as a simple knot, so it is an isotopy invariant.
It and the winding number together determine a sweep-out of arcs which in turn specify the isotopy type of $K$.
If $K$ is a strict 1-bridge braid, then it does not include as a simple knot in any lens space of order $|\omega(K)|$ or less.
However, every torus knot $T(q,\omega) \subset S^1 \times D^2$ and exceptional cable thereof includes as an unknot in the lens space obtained by $(\omega/q)$-filling on the outer torus.
\end{proof}

\begin{figure}
\includegraphics[width=4in]{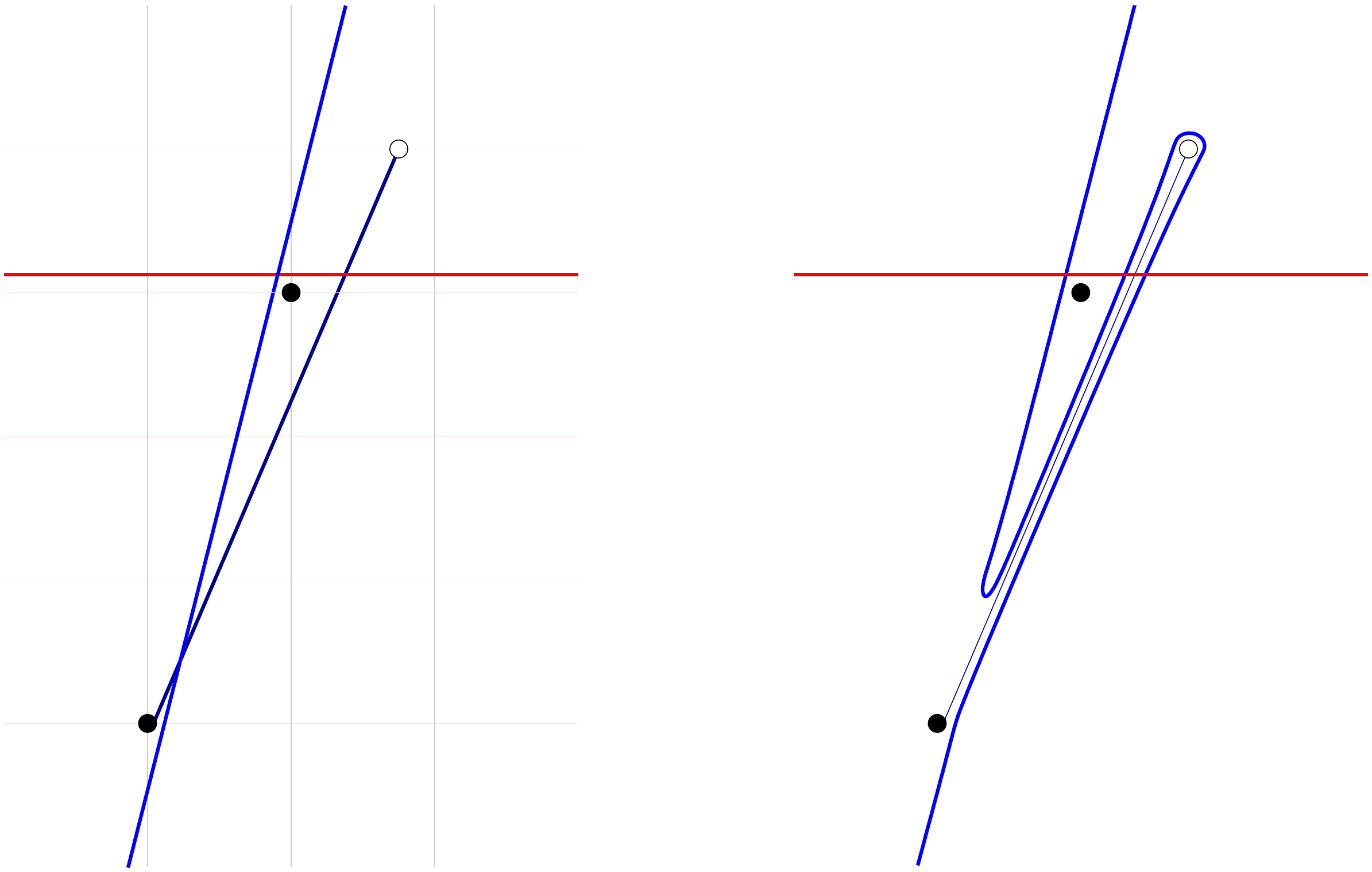}
\put(-310,170){(a)}
\put(-140,170){(b)}
\caption{Capturing a lattice point.}
\label{fig:latticepoint}
\end{figure}


\subsection{Solid torus fillings.}
\label{ss:torus}

Gabai proved that a knot in the solid torus with a non-trivial solid torus surgery is a 1-bridge braid in \cite{Gabai1990}, and Berge classified them in \cite{Berge1991}.
Berge found that, up to mirroring, $K(7,4,2)$ is the unique strict 1-bridge braid that admits more than one such surgery.
Its exterior is known as the Berge manifold.
Menasco and Zhang studied 1-bridge braids whose exteriors admit a solid torus filling on the {\em outer} torus in \cite{MenascoZhang2001}, and Wu classified them in \cite{wu2004}.

We indicate a line of approach towards Wu's result using Theorem~\ref{thm:1-bridgerefinement}.
Suppose that $K$ is a 1-bridge braid in $S^1 \times D^2$ and $(p/q)$-filling on its outer torus is a solid torus; equivalently, the inclusion of $K$ into $L(p,q)$ is a knot with solid torus exterior. 
A knot in $L(p,q)$ has a solid torus exterior if and only if it is simple and homologous to the oriented core of a Heegaard solid torus.
Since $[K]$ equals $\omega$ times a generator of $H_1(S^1 \times D^2;\bZ)$, the uniqueness of genus one Heegaard splittings of $L(p,q)$ implies that either
\begin{equation}
\label{eq:congruence}
(a) \quad \omega \equiv \pm 1 \pmod p \quad \mathrm{or} \quad (b) \quad \omega \cdot q \equiv \pm 1 \pmod p;
\end{equation}
the $2 \times 2$ possibilities correspond to the two Heegaard solid tori and the two orientations.
Therefore, Theorem \ref{thm:1-bridgerefinement} reduces the problem Wu solved to one about lattice points.
However, it appears to require considerable effort to extract Wu's result from it.
Nevertheless, we can quickly derive the following characterization of the Berge manifold:

\begin{prop}
\label{prop:unique}
Up to mirroring, the knot $K(7,4,2)$ is the unique strict 1-bridge braid whose exterior admits three distinct solid torus fillings on the outer torus.
\end{prop}

\noindent
Wu points out that this result follows from the work of Berge and Gabai, which is an amusing exercise \cite[\S 4(1)]{wu2004}.
By contrast, we deduce Proposition \ref{prop:unique} from Theorem \ref{thm:1-bridgerefinement} and the uniqueness of genus one Heegaard splittings of lens spaces.

Given linearly independent $v, w \in \bZ^2$, let $\Delta(v,w)$ denote the triangle with vertices $0$, $v$, $w$.
It is {\em empty} if it contains no lattice points besides its vertices, i.e. $\{v,w\}$ is a basis of $\bZ^2$.

\begin{proof}
Suppose that $K \subset S^1 \times D^2$ is a strict 1-bridge braid with winding number $\omega$ and slope interval $[b/d,a/c] \subset (\omega/(m+1),\omega/m)$, $m \in \bZ$.
Let $X$ denote the exterior of $K$.
Suppose that $(p/q)$-Dehn filling on the outer torus of $X$ is a solid torus.
Thus, $p/q \in [b/d,a/c]$, and one of the congruences in \eqref{eq:congruence} holds.

If \eqref{eq:congruence}(b) holds, then there must exist $s \in \bZ$ so that $\Delta((s,\omega),(p,q))$ is empty.
If $p/q \in (b/d,a/c)$, then this triangle contains $(c,a)$ if $s \le m$ and $(d,b)$ if $s \ge m+1$. 
Therefore, we must have $p/q \in \{ a/c, b/d \}$.
Furthermore, $s=m$ if $p/q = a/c$ and $s=m+1$ if $p/q = b/d$.

If instead $p > \omega$, then it follows that \eqref{eq:congruence}(a) holds, and we obtain $\omega = p-1$.
Furthermore, $(q,\omega+1) = (q,p)$ is in the cone bounded by the rays from $(0,0)$ through $(d,b)$ and $(c,a)$.
Since $m+1 \le \omega$, it follows that $(q,\omega+1) = (m+1,\omega+1)$.
Moreover, $(m+1)/(\omega+1)$ is the {\em mediant} of $d/b$ and $c/a$, meaning that $a+b = \omega+1$ and $c+d = m +1$.

Hence, if $X$ has three distinct solid torus fillings, then they have slopes $a/c$, $b/d$, and $(a+b)/(c+d)$ (in conformity with the cyclic surgery theorem \cite{Culler1987CyclicSurgery}).
We assume this going forward.

Suppose that \eqref{eq:congruence}(b) holds for both $b/d$ and $a/c$.
It follows that all lattice points interior to $\Delta((m,\omega),(m+1,\omega))$ fall on the rays generated by $(d,b)$ and $(c,a)$.
In particular, $(m,\omega-1)$ is a multiple of exactly one of these two points, meaning that $\omega \equiv 1 \pmod e$ for a unique value $e \in \{ a,b \}$.
Thus, \eqref{eq:congruence}(a) holds for whichever of $b/d$ and $a/c$ has numerator different from $e$.

On the other hand, if \eqref{eq:congruence}(a) holds for some $p/q \in \{a/c,b/d\}$, then it must hold with the sign $+1$, since otherwise one of $a$ or $b$ divides the other, a contradiction.
Since $(m,\omega-1)$ is the unique lattice point in $\Delta((m,\omega),(m+1,\omega))$ with its $y$-coordinate, it follows that it is a multiple of $(q,p)$.
Thus, \eqref{eq:congruence}(a) holds for at most one of $a/c,b/d$.

In total, \eqref{eq:congruence}(a) holds for one of $a/c$, $b/d$ and \eqref{eq:congruence}(b) holds for the other.
Applying the linear map $(x,y) \mapsto (y-x,y)$ exchanges $K$ with its mirror (see \cite[\S 4]{wu2004}).
Thus, we may assume that  \eqref{eq:congruence}(a) holds for $a/c$ and \eqref{eq:congruence}(b) for $b/d$.
We have $k \cdot (c,a) = (m,\omega-1)$ for some $k \in \bZ$, $k > 0$.
Since $(m+1,\omega+1) = (c,a) + (d,b)$, we obtain $(m+1,\omega+1) = k \cdot (c,a) + (1,2)$, and so $(k-1) \cdot (c,a) + (1,2) = (d,b)$.
Hence
\[
1 = \left| \begin{matrix} d & b \\ c & a \end{matrix} \right| = \left| \begin{matrix} 1 & 2 \\ c & a \end{matrix} \right| = a-2c 
\quad \mathrm{and} \quad
1 = \left| \begin{matrix} m+1 & \omega \\ d & b \end{matrix} \right| = \left| \begin{matrix} m+1 & \omega+1 \\ d & b \end{matrix} \right| + d = -1+d.
\]
We deduce in turn that $d=2$, $k=2$, $c=1$, $a=3$, $b=5$, $\omega = 7$, and $m = 2$.
Since $(7/3, 5/2, 3/1, 7/2)$ is the Farey sequence of fractions between $\omega/(m+1)$ and $\omega/m$, the slope interval of $K$ is $[5/2,3/1]$, and it has winding number $7$.
See Figure \ref{fig:724plane}.
These invariants specify $K \simeq K(7,4,2)$.
\end{proof}


\bibliographystyle{amsalpha2}

\bibliography{Reference}

\end{document}